\documentclass[12pt]{amsart}
\usepackage{amsmath,amsfonts,amssymb,amsthm,graphicx,thmtools,mathtools,xparse,algorithmic,enumitem}
\usepackage{ulem}  
\PassOptionsToPackage{hyphens}{url}\usepackage{hyperref}
\hypersetup{
    colorlinks,
    citecolor=black,
    filecolor=black,
    linkcolor=black,
    urlcolor=black
}

\makeatletter
\g@addto@macro\normalsize{%
  \setlength\abovedisplayskip{10pt}
  \setlength\belowdisplayskip{10pt}
  \setlength\abovedisplayshortskip{10pt}
  \setlength\belowdisplayshortskip{10pt}
}
\makeatother

\renewcommand{\emph}[1]{{\it #1}}

\newcommand{\NN}{\mathbb{N}}
\newcommand{\ZZ}{\mathbb{Z}}
\newcommand{\QQ}{\mathbb{Q}}
\newcommand{\RR}{\mathbb{R}}
\newcommand{\CC}{\mathbb{C}}
\newcommand{\PP}{\mathbb{P}}

\newcommand{\calN}{\mathcal{N}}
\newcommand{\calO}{\mathcal{O}}
\newcommand{\variety}{\mathcal{V}}

\newcommand{\bs}{\boldsymbol}

\DeclareMathOperator{\Vol}{\rm Vol}

\DeclareMathOperator{\ind}{\rm ind}

\DeclareMathOperator{\Cone}{\rm Cone}

\DeclareMathOperator{\spa}{\rm span}
\DeclareMathOperator{\conv}{\rm conv}

\DeclareMathOperator{\tr}{{\rm tr}}
\DeclareMathOperator{\codim}{{\rm codim}}
\DeclareMathOperator{\ess}{{\it ess}}

\DeclarePairedDelimiter{\norm}{\lVert}{\rVert}
\NewDocumentCommand{\normL}{ s O{} m }{%
  \IfBooleanTF{#1}{\norm*{#3}}{\norm[#2]{#3}}_{L_2(\Omega)}%
}

\newcommand\diff[3]{(D^{#2} \, #1)_{#3}}

\newtheorem{thm}{Theorem}[section]
\newtheorem{prop}[thm]{Proposition}
\newtheorem{cor}[thm]{Corollary}
\theoremstyle{definition}
\newtheorem{example}[thm]{Example}
\newtheorem{defn}[thm]{Definition}
\newtheorem{remark}[thm]{Remark}
\declaretheorem{algorithm}

\newcommand{\sI}{\includegraphics{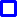}}
\newcommand{\sII}{\includegraphics{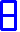}}
\newcommand{\sT}{\includegraphics{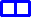}}

\newcommand{\defcolor}[1]{{#1}} 
\newcommand{\demph}[1]{\emph{#1}}

\title[Certification for Polynomial Systems via Square Subsystems]{Certification for Polynomial Systems\\ via Square
  Subsystems} 
\author{Timothy Duff}
\address{School of Mathematics\\
         Georgia Institute of Technology\\
         686 Cherry Street\\
         Atlanta, GA 30332-0160\\
         USA}
\email{tduff3@gatech.edu}
\urladdr{http://people.math.gatech.edu/~tduff3/}
\author{Nickolas Hein}
\address{Department of Mathematics and Computer Science\\
         Benedictine College\\
         1020 N.\ 2nd St\\
         Atchison, KS 66002\\
         USA}
\email{nhein@benedictine.edu}
\urladdr{https://www.benedictine.edu/faculty-staff/hein-nickolas}
\author{Frank Sottile}
\address{Frank Sottile\\
         Department of Mathematics\\
         Texas A\&M University\\
         College Station\\
         Texas \ 77843\\
         USA}
\email{sottile@math.tamu.edu}
\urladdr{www.math.tamu.edu/\~{}sottile}
\thanks{Work of Sottile supported in part by the National Science Foundation under grant DMS-1501370}
\thanks{Work of Duff supported in part by the National Science Foundation under grant DMS-1719968}
\thanks{Duff and Sottile supported by the ICERM}
\subjclass[2010]{65G20, 65H10}
\keywords{certified solutions, alpha theory, polynomial system, numerical algebraic
  geometry, Newton-Okounkov bodies, Schubert calculus}

\begin{document}
\begin{abstract} 
 We consider numerical certification of approximate solutions to a system of polynomial equations with more equations than
 unknowns by first certifying solutions to a square subsystem. 
 We give several approaches that certifiably select which are solutions to the original
 overdetermined system.
 These approaches each use different additional information for this certification, such as liaison, Newton-Okounkov
 bodies, or intersection theory.
 They may be used to certify individual solutions, reject nonsolutions, or certify that we have found all solutions.
\end{abstract}

\maketitle
\section{Introduction}
\label{sec:intro}

Given polynomials $f=(f_1,\dotsc,f_N)$ with $f_i\in\CC [z_1, \dotsc , z_n]$, 
an approximate solution to the system $f_1(z)=\dotsb=f_N (z)=0$ is an  
estimate $\hat{\zeta}$ of some point $\zeta$ where the polynomials all vanish ($\zeta$ is a \emph{solution} to $f$), such
that the approximation error $\|\zeta - \hat{\zeta}\|$ can be refined efficiently as a function of the input size and
desired precision. 
Numerical certification seeks criteria and algorithms 
that guarantee that a computed estimate $\hat{\zeta}$ of a solution $\zeta$ to $f$ is an
approximate solution in this sense.

Many existing certification methods~\cite{Kraw,Sma86} are for square systems, where $N=n$.
These exploit that the isolated, nonsingular solutions to the system are exactly the fixed points of the Newton
operator $N_f \colon \CC^n \to \CC^n$ given (where defined) by 
 \begin{equation}\label{Eq:Newton}
   N_f(z)\ :=\ z - Df(z)^{-1} f(z)\,,
 \end{equation}
where $Df(z)$ is the Jacobian matrix of the system $f$ evaluated at $z$.
A Newton-based certificate establishes that the sequence of Newton iterates $(N_f^k(\hat{\zeta})\mid k\in\NN)$ converges
to a solution $\zeta$ to $f$. 
Examples include both Smale's $\alpha$-test~\cite{SS93,Sma86} (typically performed in rational arithmetic)
and Krawczyk's method~\cite{Kraw} (based on interval arithmetic).

Once such a certificate is in hand, we say that $\hat{\zeta}$ is an approximate solution to $f$ with associated solution
$\zeta$.  
Further refinements bound the distance to the associated solution $\|\zeta - \hat{\zeta}\|$, decide if two approximate
solutions are associated to the same solution, and, in the case of real systems, decide if the associated solution is
real~\cite{HS12}. 

Certification in the overdetermined case, where $N>n$, poses challenges not encountered in the square case. A detailed study of the ``least-squares'' Newton operator,
\begin{equation}\label{Eq:ODNewton}
   N_{f, \dagger }(z)\ :=\ z - Df(z)^{\dagger} f(z)\,,
\end{equation}
was undertaken by Dedieu and Shub~\cite{DS00}. Here, ``$\dagger $'' indicates the Moore-Penrose pseudoinverse. Among the contributions in their work is a generalization of Smale's $\alpha$-theorem, Theorem 3 in~\cite{DS00}, giving sufficient conditions for the convergence of $(N_{f, \dagger }^k(\hat{\zeta})\mid k\in\NN)$ to a fixed point. Since there may be many fixed points of $N_{f, \dagger}$ which are not solutions to $f,$ this criterion is generally not sufficient to certify that $\hat{\zeta }$ is an approximate solution to $f.$ In fact, the fixed points of $N_{f,\dagger}$ are precisely the solutions to the square system defined by
\begin{equation}\label{Eq:ODNewtonCrit}
g_i (z) := \partial_{z_i} (f_1^2 + \cdots + \cdots + f_N^2) = 0
\, \, \, \, \, \, \, \, \, \, 
(1 \le i \le n).
\end{equation}

Clearly the solutions to $g$ include the solutions to $f$: we thus say that $g$ is a \emph{square subsystem} of $f.$ Our point of view in this paper is that certification of approximate solutions may be possible if we are given a square subsystem of $f$---typically \emph{not} the system in Equation~\ref{Eq:ODNewtonCrit}---together with some global information about the excess solutions.

Alternate approaches to certification in the overdetermined case have been considered in previous work. In~\cite{AHS18} a hybrid symbolic-numeric approach is used when the polynomials in $f$ have rational coefficients. This requires computing an exact rational univariate representation~\cite{rouillier} and using that to certify approximate solutions. Alternatively, one may attempt to lift the approximation $\hat{\zeta}$ to an approximate solution to a square system in \emph{more} variables. This is the approach taken for Schubert problems in~\cite{HHS16,HS17}.

Though reduction to the case of a square subsystem is a natural idea, a suitable square subsystem $g$ and the requisite global information are not easily obtained in general. Rather than prescribing a single approach, we follow the pattern of this reduction through a series of algorithms and examples. We highlight how abstract tools such as liaison theory and Newton-Okounkov bodies may be brought to bear on certification, and illustrate certification for problems of interest in the Schubert calculus and computer vision.

\begin{remark}
\label{remark:family}
It is well-understood that perturbing an overdetermined system, eg.~by adding generic constants, will generally produce an
inconsistent system. Nevertheless, there are many \emph{families} of overdetermined polynomial systems such that a generic
member of the family has finitely many isolated solutions. In other words, a family of overdetermined systems need not be
\emph{over-constrained}.
In general, a family of polynomial systems specified by some parameters $p\in \CC^m$ may be understood via the incidence variety
$$
V_f = \{ (z,p ) \in \CC^n \times \CC^m \mid f(z,p) = 0 \}.
$$
The family is \emph{well-constrained} if the projection onto $\CC^m$ is dominant and $\dim V_f =m.$ The examples in Sections~\ref{subsec:schubertExample} and~\ref{subsec:vision} fit naturally into the category of well-constrained families of overdetermined systems, where it is reasonable to seek certificates even for generic $p.$
\end{remark}


The algorithms in our paper address the problems below.\medskip

\noindent{\bf Problem 1.}
How may we certify that  a point $\zeta\in\CC^n$ is an approximate solution to $f$?\medskip

\noindent{\bf Problem 2.}
Suppose it is known that $f$ has $e$ solutions. 
How may we certify that a set $Z\subset\CC^n$ of $e$ points consists of approximate solutions to $f$?\medskip

In Section~\ref{sec:appxSols}, we recall various notions of ``approximate solution'' that have been considered in the literature; our Definitions~\ref{defn::appx-solution} and~\ref{defn::effective-appx-solution} encompass these various notions and apply just as well to the overdetermined case.
We also explain a simple exclusion criterion that indirectly certifies particular approximate solutions to a subsystem $g$ of $f$ as non-solutions to $f$.
In Section~\ref{sec:nonsol}, we explain how global information about the excess solutions to $g$ may be used to certify approximate solution to $f$ in the sense of Section~\ref{sec:appxSols}, thus solving Problems 1 and 2.
We also give an alternative approach for Problem 2 that incorporates global information about $f,$ and explain how the global information about $g$ can be calculated in terms of Khovanskii bases and their associated Newton-Okounkov bodies in Section~\ref{sec:nobody}.
Section~\ref{sec:liaison} discusses one further approach to Problem 1 that is based on liaison theory.
In Section~\ref{sec:examples}, we give three examples illustrating our algorithms.
One involves a finite Khovanskii basis, another is from the Schubert calculus, and a third is from computer vision.

\section{Approximate solutions}
\label{sec:appxSols}
Throughout, we will fix positive integers $n\leq N$.
All polynomials will lie in the ring $\CC[x_1,\dotsc,x_n]$.
We will  write $f$ for a system $f_1,\dotsc,f_N$ of $N$ polynomials.
The system $f$ is \demph{square} when $N=n$.

\begin{defn}
\label{defn::appx-solution}
A \demph{$\rho$-approximate solution} to a (possibly overdetermined) polynomial system $f$ is a triple
$(\hat{\zeta}, \rho , \calN_f)$, where $\hat{\zeta }\in \CC^n$,
$\rho \in \RR_{>0}$, and $\calN_{f}\colon U \to \CC^n$ is a map defined on some $U \subset \CC^n$ such that
\begin{itemize}
\item[1)] There exists $\zeta \in \variety (f)$ such that $\norm{\zeta - \hat{\zeta} } < \rho $, and
\item[2)] all iterates $\calN_f^k (\hat{\zeta} )$ are defined and the sequence $\calN_f^k (\hat{\zeta} )$
  converges to $\zeta$ as $k\to\infty$.
\end{itemize}
\end{defn}
Here $\norm{\cdot}$ indicates the usual Hermitian norm on $\CC^n$.
We will refer to $\hat{\zeta}$ as an approximate solution when the procedure $\calN_f$
and constant $\rho$ are understood.
We call the point $\zeta \in \variety (f)$ in (1) the solution to $f$ \demph{associated} to $\hat{\zeta}$. We sometimes refer to $\mathcal{N}_f$ as a \demph{refinement operator}. This is typically some incarnation of Newton's method.

\begin{remark}
  \label{remark::field}
In our examples, the system $f$ and the approximate solution
  $\hat{\zeta }$ are defined over  the rationals $\mathbb{Q}$ or the Gaussian rationals $\QQ [\sqrt{-1}]$.
  Since numerical solvers typically output floating point results, care must be taken to control rounding errors when
  computing certificates.
  One option for certification is to perform all subsequent operations in rational arithmetic.
  Interval and ball arithmetic give yet another approach (discussed in Subsection~\ref{SS:other}).
  A ``certificate'' obtained without controlling rounding errors may still be of practical value.
  Following~\cite{HS12}, we  call this a \demph{soft certificate}.  
\end{remark}

  \begin{remark}
  \label{remark::computability}
In practice, the map $\calN_f$ in Definition~\ref{defn::appx-solution} should restrict to a computable
   function $\calN_{f,\QQ}\colon U \cap \QQ[\sqrt{-1}]^n \to \QQ[\sqrt{-1}]^n$.
   Our algorithms assume an oracle for $\calN_f,$ and we generally take a naive approach to questions of computability and complexity. However, we do not rely on special features of
   nonstandard models of computation such as the Blum-Shub-Smale machine~\cite{BCSS}. 
\end{remark}



Let $g$ be a square sytem and $(\hat{\zeta},\rho,\calN_g)$ be an approximate solution to $g$ with
associated solution $\zeta$.
Our main concern is to certify that $\zeta \in \variety (f)$
when $g$ is a square subsytem of $f$---a seemingly difficult task {\it a priori.}
It is however relatively simple to certify that $\zeta $ is {\it not} a solution to
a single polynomial $f$, provided that $\rho $ is sufficiently small.
For $k\in\NN$, let $S^k\CC^n$ be the $k$th symmetric power of $\CC^n$.
This has a norm $\norm{\cdot}$ dual to the standard unitarily invariant norm on homogeneous polynomials, and which
satisfies $\|z^k\|\leq \|z\|^k$, for $z\in\CC^n$. 
The $k$-th derivative of $g$ at $\zeta$ is a linear map $\diff{g}{k}{\zeta}: S^k(\CC^n) \to \CC^n$ with operator norm,
 \begin{equation}
  \label{op_norm}
   \norm{ \diff{g}{k}{\zeta} }\ :=\  \displaystyle\max_{
    \substack{
      w \in S^k\CC^n
      \\[0.25em]
      \norm{w}=1}
      }
    \norm{\diff{g}{k}{\zeta} \, (w)}\,.
\end{equation}

\begin{prop}
  \label{prop:reject}
  Suppose that $(\hat{\zeta }, \rho , \calN_g)$ is an approximate solution to a square polynomial system $g$ with
  associated solution $\zeta$.
  For any polynomial $f$, if 
  \begin{equation}
    \label{eq:taylor_bound}
      |f (\hat{\zeta })| - \displaystyle\sum_{k=1}^{\deg f_j}
      \displaystyle\frac{\norm{\diff{f}{k}{\hat{\zeta}}}}{k!}
      \cdot \rho^k
      \ >\ 0\,,
  \end{equation}
then $f(\zeta)\neq 0$.
\end{prop}
\begin{proof}
By Taylor expansion, it follows that $f(z)\ne 0$ for any $z\in B(\hat{\zeta} , \rho )$.
\end{proof}

Let us write \defcolor{$\delta(f,g, \hat{\zeta})$} for the difference in the inequality~\eqref{eq:taylor_bound}, which we
will call a \demph{Taylor residual}.
Note the implicit dependence on $\rho $ in Definition~\ref{defn::appx-solution}.
If $f=(f_1,\dotsc,f_N)$ is a polynomial system, then we define its Taylor residual $\delta(f,g, \hat{\zeta})$ to be
maximum of the Taylor residuals $\delta(f_i,g, \hat{\zeta})$, for $i=1,\dotsc,N$.
For this test of nonvanishing using Taylor residuals to be practical, we need to estimate the operator norms of the
higher derivatives.
One possible bounding strategy,  as explained in~\cite[\S I-3]{SS93} and~\cite[\S1.1]{HS12}, uses the first derivative
alone.
Another option, less suitable for polynomials of high degree, is to bound with the entry-wise $\ell_2$ or $\ell_1$ norms of these tensors.

A consequence of Definition~\ref{defn::appx-solution} is that each iterate $\calN_f^k (\hat{\zeta})$ is an approximate
solution, as $\calN_f^k (\hat{\zeta }) \to \zeta $.
We wish to quantify this rate of convergence.
The triangle inequality gives a test for when approximate solutions $(\hat{\zeta_1}, \rho_1 , \calN_f)$ and
$(\hat{\zeta_2 }, \rho_2, \calN_f)$ have distinct associated solutions, namely if 
\begin{equation}
  \label{eq::distinct}
  \norm{\hat{\zeta_1} - \hat{\zeta_2}}\ >\ \rho_1 + \rho_2\,.
  \end{equation}
  It is useful to have some additional criterion when two approximate solutions have the {\it same} associated solutions,
  that is, we wish to certify {\it uniqueness} of the associated solution in a sufficiently small region.
  This motivates our next definition.


\begin{defn}
\label{defn::effective-appx-solution}
An \emph{effective approximate solution} $(\hat{\zeta}, \calN_f, \rho, k_*)$ to a system $f$ consists of a weakly decreasing
rate function $\rho \colon \mathbb{N} \to \mathbb{R}_{>0}$ with
$\lim_{k\to\infty}\rho(k)=0$, and an integer $k_*$ such that\vspace{-5pt} 
\begin{itemize}
\item[1)] $\hat{\zeta}$ is a $\rho (0)$-approximate solution to $f$ with associated solution $\zeta $,
\item[2)] $\norm{\calN_f^{j} (\hat{\zeta }) - \zeta } < \rho (k)$ for all $j\ge k$, and
\item[3)] For some iterate $k_*$, $\zeta $ is the unique solution in the ball $B\left(\calN_f^{(k_*)} (\hat{\zeta }),\,  2\, \rho (k_*)\right)$.
  \end{itemize}
We say the rate of convergence for the effective approximate solution has order $\rho (k)$.
\end{defn}

The rate of convergence is \emph{quadratic} when  $\rho (k) = 2^{- 2^{O(k)} } \, \norm{\zeta - \hat{\zeta } }$.
This implies that each application of $\calN_f(\cdot)$ roughly doubles the number of significant digits in $\hat{\zeta}$.
We generalize the method for certifying distinct solutions in \cite[\S I-2]{HS12}.

\begin{prop}
  \label{prop:separating}
Given a set of effective approximate solutions $S'=\{ (\hat{\zeta_i}, \calN_f, \rho_i, k_*^i) \}$ to a system $f$, we
  may compute a set $S$ of refined approximate solutions with distinct associated solutions comprising all solutions
  associated to the set $S'$. 
\end{prop}

\begin{proof}
We need only replace each $\hat{\zeta_i}$ with its refinement $\calN_f^{k_*^i} (\hat{\zeta_i})$. After refinement, the
solutions associated to $\hat{\zeta_i}$ and $\hat{\zeta_j}$ are distinct if and only if inequality~\eqref{eq::distinct}
holds. 
\end{proof}

Proposition~\ref{prop:reject} may fail to certify that an extraneous solution $\zeta$ is not a solution to $f$.  However, if $\hat{\zeta}$ is an effective approximate solution for $\mathcal{N}_g ,$ then this test will succeed after sufficiently many refinements. 

%
\begin{cor}\label{cor:refine_and_reject}
Let $f=(f_1,\dotsc,f_N)$ be a system of polynomals, and suppose $(\hat{\zeta }, \mathcal{N}_g, \rho , k_*)$ is an effective approximate solution such that the associated solution $\zeta\not\in\variety(f)$.
There is a $k\geq 0$ such that the Taylor residuals $\delta\left(f,g, \calN_g^i(\hat{\zeta})\right)$ are positive for all $i\geq k$.
\end{cor}  

\begin{proof}

Since $\mathcal{N}_g^j (\hat{\zeta })\to \zeta,$ we may argue that $\delta \left(f, g, \mathcal{N}_g^k (\hat{\zeta } )\right) > 0$ for some $k$ by Taylor expansion as in Proposition~\ref{prop:reject}. Since $\rho (\cdot )$ is weakly decreasing, it follows from 2) in Definition~\ref{defn::effective-appx-solution} that the Taylor residuals remain positive for all $i\ge k.$
\end{proof}  

Certificates for square systems are generally based on Newton's method.
We now observe that Definitions~\ref{defn::appx-solution} and~\ref{defn::effective-appx-solution} encapsulate several
existing certification paradigms for square systems. 

\subsection{Smale's $\alpha$-theory}\label{SS:alpha}
The central quantities of Smale's $\alpha$-theory are defined as follows.  
With $g$ as above and $\hat{\zeta}\in \CC^n$ a point where $Dg(\hat{\zeta})$ is invertible,
\begin{align} 
  \alpha (g,\hat{\zeta})\ &:=\ \beta(g,\hat{\zeta}) \cdot \gamma (g,\hat{\zeta})\  \text{, \ where}\nonumber\\
  \beta (g,\hat{\zeta})\ &:=\ \norm{\hat{\zeta} - N_g (\hat{\zeta})}
                           \ =\  \norm{ Dg(\hat{\zeta})^{-1} g(\hat{\zeta})}\  \text{, \  and}\\
  \gamma (g,\hat{\zeta})\ &:=\ \nonumber
                 \sup_{k\geq 2}\left\|\frac{Dg(\hat{\zeta})^{-1}\diff{g}{k}{\hat{\zeta}} }{k!}\right\|^{\frac{1}{k-1}}\  .
\end{align}
Note that $\beta(g,\hat{\zeta})$ is the length of a Newton step at $\hat{\zeta}$.
The following proposition gives a criterion for approximate solutions in the sense of Definition~\ref{defn::appx-solution}.

\begin{prop} [{\cite[p.~160]{BCSS}}]
  \label{prop:alpha}
  Let $g$ be a square polynomial system and $\hat{\zeta}\in\CC^n$.
  If
  \[
    \alpha(g,\hat{\zeta})\ <\ \frac{13-3\sqrt{17}}{4}\ \approx\ 0.15767078\,,
  \]
  then $\hat{\zeta}$ is $2 \beta (g, \hat{\zeta})$-approximate solution to $g$ and the Newton iterates $N_g^k (\hat{\zeta })$
  converge quadratically. 
\end{prop}

Criteria for quadratically convergent effective approximate solutions in the sense of
Definition~\ref{defn::effective-appx-solution} can also be given in terms of $\alpha (g, \hat{\zeta })$.
The analysis amounts to showing that $N_g$ is a contraction mapping in a suitable neighborhood of $\zeta$.
This is given by the ``robust'' $\alpha$-theorem (Theorem 6 and Remark 9 of~\cite[Ch.~8]{BCSS}).

\begin{prop}
  \label{prop:same}
  Let $g$ be a square polynomial system and $\hat{\zeta}\in\CC^n$ an approximate solution to $g$ with associated solution
  $\zeta$ and suppose that $\alpha(g,\hat{\zeta})<0.03$.
  If $\hat{\zeta}'\in \CC^n$ satisfies
  \[
    \| \hat{\zeta} - \hat{\zeta}'\|\ <\ \frac{1}{20\gamma(g,\hat{\zeta})}\,,
  \]  
  then $\hat{\zeta}'$ is an
  approximate solution to $g$ with associated solution $\zeta$. 
\end{prop}

It follows that, for $\rho (k) = 2^{-2^{k-1}} \beta (\hat{\zeta})$, we have that $(\hat{\zeta }, N_g, \rho , 0)$ is an
effective approximate solution in the sense of Definition~\ref{defn::effective-appx-solution}.

\subsection{Other approaches}\label{SS:other}

The classical analysis of Newton's method is due to Kantorovich~\cite{Kan82}.
Several variations of Kantorovich's theorem exist, typically assuming some local Lipchitz condition on the Jacobian $D_g$ and boundedness conditions on
$D_g (\hat{\zeta })^{-1}$.
Certificates based on Kantorovich's theorem thus rely on \emph{a priori} bounds in a region containing $\hat{\zeta}.$
Explicit bounds on the rate of convergence in terms of the Lipchitz and bounding constants
are given in various works~\cite{Tap71,TG74,Deuf11}.  
We refer to~\cite{LS16} for a survey of variants and an explanation of the relationship between Kantorovich's theorem and $\alpha $-theory.

Approximate solutions may also be understood within the general program of interval and ball arithmetics.
Both paradigms rely on defining arithmetic operations on intervals or balls and are definable in
either exact or floating point arithmetic.
In general, operations on intervals represent \demph{enclosures}.
In exact interval arithmetic, we define the sum by $[a,b]+[c,d]=[a+c,b+d]$.
For floating point arithmetic, we may either accept a soft certificate or control rounding errors when defining arithmetic
operations so as to obtain a rigorous certificate.
We refer to~\cite{Kul14,MKRC09,Tuc11} for a more comprehensive treatment of these notions.
A variety of interval/ball-valued Newton iterations have been studied.
A popular variant is the Krawcyzk Method---see~\cite[Chapter 6]{MKRC09} for an introduction,~\cite{MJ77} for quadratic
convergence, and~\cite{BLL19} for extensions to complex analytic functions.
Once a Newton-like iteration is in place, we get criteria for approximate solutions in the sense of
definitions~\ref{defn::appx-solution} and~\ref{defn::effective-appx-solution} by taking $\hat{\zeta}$ to be the midpoint/center of the enclosing interval/ball.


\section{Certification via nonsolutions}\label{sec:nonsol}
In this section we consider certification in the setting where we have an overdetermined system given by $f_1,\ldots , f_N\in \CC [x_1,\ldots , x_n],$ a full set of approximate solutions to some square subsystem $g,$ and prior knowledge of
an integer $d$ such that
\begin{equation}
\label{eq:bound}
  d\ =\ \# \left( \variety (g) \setminus \variety (f)\right)\,.
\end{equation}

From this information, the Newton operator $N_g$ can be used to give an approximate solution to $f$ in the sense of
Definition~\ref{defn::appx-solution}. 
We make this precise in Section~\ref{sec:algs}; Algorithm~\ref{alg:ind} provides one possible solution to Problem 1 from
the introduction.
For Problem 2, we give an essentially different approach (Algorithm~\ref{alg:set}) that assumes knowledge of the number of 
solutions to $f$. 

Knowledge of $d$ may come from rigorous mathematical proof (eg.~the examples in~\ref{subsec:schubertExample}) or by some form of certified computation.
If all points in $\variety (g)$ are isolated, then $d$ is simply the degree of the saturated ideal $\langle g \rangle : \langle f \rangle^\infty $.
Thus, if we can compute Gr\"{o}bner bases for the ideals generated by both polynomial systems, then we have $d$ which is an
admissible input for Algorithms~\ref{alg:ind} and~\ref{alg:set}, which are given in this section.

Aside from addressing Problems 1 and 2, we explain another approach to computing $d$ in the special case
where $g$ is obtained by ``squaring up'', or \demph{randomization}~\cite{WS05}.
This means we have a suitably generic $n\times N$ matrix $A\in \CC^{n\times N}$ such that
\begin{equation}
\label{eq:mat}
  g\ :=\ 
\left( \begin{array}{c}
g_1(z)\\
\vdots \\
g_n(z)
\end{array}\right)\  =\ A \: \left( \begin{array}{c}
f_1(z)\\
\vdots \\
f_N(z)
\end{array}\right)\ =\ 0\,.
\end{equation}

For such $g,$ the number $d$ can sometimes be computed from Khovanskii bases (a generalization of SAGBI bases) for the
polynomial algebra in $n{+}1$ variables given by $\CC [t f_1, \ldots , t f_N]$.
We give an overview of this theory in Section~\ref{sec:nobody}.
An attractive feature of Khovanskii bases is that we may, in principle, work with the algebra $\CC [t f_1, \ldots , t f_N]$ itself rather than some presentation $\CC [x_1,\ldots, x_N]\to \CC [t\, f_1, \ldots , t\, f_N]$.
There is, however, a significant trade-off, which is that finite Khovanskii bases need not exist. Nevertheless, we feel that computation of Khovanskii bases deserves to be more thoroughly explored.
Certification is a particular application which may benefit from more efficient and robust computational tools for Khovanskii bases. 

\subsection{Certification algorithms}\label{sec:algs}

We now formulate our first algorithm for solving Problem 1. This is the content of Theorem~\ref{thm:ind}, based on Definition~\ref{defn::appx-solution}. Note that Algorithms~\ref{alg:ind} and~\ref{alg:set} assume pairwise distinct approximate solutions and explicit separating balls, respectively. By Proposition~\ref{prop:separating}, it is enough to require that the $\hat{\zeta_i}$ are effective approximate zeros and apply these algorithms after refinement.

\begin{algorithm}[Certifying individual solutions]\label{alg:ind}$ $\\
  \textbf{Input:} $(f,g, d, S )$
  \begin{itemize}
  \item[] $f$ --- a polynomial system
  \item[] $g$ --- a square subsystem of $f$
  \item[] $d \in \NN$ satisfying \eqref{eq:bound}
  \item[] $S = \{ \hat{\zeta_1 }, \ldots , \hat{\zeta_m} \}$ --- pairwise distinct approximate solutions to $g$ 
  \end{itemize}
  \textbf{Output:} $T\subset S,$ a set of approximate solutions to $f$
  \begin{algorithmic}
    \STATE Initialize $R\gets \emptyset $
    \STATE{\textbf{for } $j=1,\dotsc, m$ \textbf{ if } $\delta (f,g,\hat{\zeta_j}) >0$ \textbf{ then }
              $R\gets R \cup \{ \zeta_j \}$ }
    \STATE{\textbf{if } $(\# R == d)$ \textbf{ then } $T\gets S\smallsetminus R$ \textbf{, else } $T\gets \emptyset $}
      \RETURN $T$
  \end{algorithmic}
\end{algorithm}

\begin{remark}
\label{remark:boundd}
\emph{A priori,} we only need to know that $d \ge \# \left( \variety (g) \setminus \variety (f)\right)$---if the inequality is strict, we necessarily return an empty set.
\end{remark}

\begin{thm}
  \label{thm:ind}
  Suppose that $f,g,d,S$ are valid input for Algorithm~\ref{alg:ind}.
  Then its output consists of approximate solutions to $f$.
\end{thm}
\begin{proof}
  If $T$ is empty there is nothing to prove.
  Otherwise, there are $d$ distinct solutions to $g$ associated to points of $R$---by Proposition~\ref{prop:reject}, these
  are not solutions to $f$. 
  Since the solutions associated to points of $T$ are disjoint from those associated to points of $R$, by assumption
  and~\eqref{eq:bound} they associate to solutions to $f$. 
\end{proof}


We now give a second algorithm using $\alpha$-theory to certify solutions to an overdetermined system $f$ to solve Problem 2.
Suppose that we have an overdetermined system $f$ that is known to have $e$ solutions whose square subsystems are
known to have $d$ solutions.
While we could apply Algorithm~\ref{alg:ind} to certify approximate solutions to $f$, we
propose an alternative method to solve this problem.

\begin{algorithm}[Certifying a set of solutions]\label{alg:set}$ $\\
  \textbf{Input:} $(d,e,f,g,g',S,S',B)$
  \begin{itemize}
  \item[] $e\leq d$ --- integers
  \item[] $f$ --- a polynomial system with $e$ solutions
  \item[] $g, g'$ --- two square subsystems of $f$
  \item[] $S= \{ \hat{\zeta_1}, \ldots , \hat{\zeta_d} \}$ --- a set of $d$ distinct approximate solutions to $g$
  \item[] $B= \{ B(\hat{\zeta_1}, \rho_1 ) , \,  \ldots ,\, B(\hat{\zeta_d}, \rho_d) \} $
    --- disjoint balls separating elements of $S$.
  \item[] $S'$ --- a set of $d$ distinct approximate solutions to $g'$
  \end{itemize}
  \textbf{Output:} $T\subset S$, a set of approximate solutions to $f$
  \begin{algorithmic}[1]
    \STATE Initialize $T \gets \emptyset $
\STATE   $r\gets \displaystyle\min_{1\le i < j \le d } \, \Big( \| \hat{\zeta}_i- \hat{\zeta}_j\| - (\rho_i + \rho_i) \, \Big)$
\FOR{ \indent $\hat{\zeta}'\in S'$ }
\STATE{\textbf{repeat } $\hat{\zeta }' \gets N_{g'} (\hat{\zeta} ')$ \textbf{ until } $2\, \beta (g', \hat{\zeta} ')< r/3$}
\STATE{ $\rho ' \gets 2 \, \beta (g', \hat{\zeta} ')$}
\STATE{\textbf{for } $j=1, \, \ldots , d$ \textbf{ if } $B(\hat{\zeta_j}, \rho_j)\cap B '(\hat{\zeta }', \rho ') \ne \emptyset $ \textbf{ then } $T\gets T \cup \{ \hat{\zeta_j } \}$}
\ENDFOR
    \STATE{\textbf{if } $\big(\#T == e \, \big)$, then \textbf{return } $T$ \textbf{, else }\textbf{ return } FAIL}
  \end{algorithmic}
\end{algorithm}
Note that the intersection of balls in line 6 is non-empty if and only if
\[
    \rho ' + \rho_j\ >\ \norm{\hat{\zeta_j}- \hat{\zeta}'}\ ,
\]
so this condition may be decided in rational arithmetic if a hard certificate is
desired.

\begin{thm}
  \label{thm:set}
  Let $f$ be a system of polynomials having $e$ solutions whose general square subsystems have $d$ solutions.
  Then Algorithm~\ref{alg:set} either returns FAIL or it returns a set $T$ of approximate solutions to $f$ whose associated
  solutions are all the solutions to $f$.
\end{thm}

As with Algorithm~\ref{alg:ind}, while the hypotheses appear restrictive, they are natural from an intersection-theoretic
perspective, and are satisfied by a large class of systems of equations. We explain one such family coming from Schubert calculus in Section~\ref{subsec:schubertExample}.

\begin{proof}
  Since the balls $B(\hat{\zeta}_i,\rho_i)$ are pairwise disjoint, the quantity $r$ is positive.
  Thus the refinement of each approximate solution $\hat{\zeta}'$ on line 4 terminates.
  Having refined each $\hat{\zeta } ' \in S',$ note that $B(\hat{\zeta } ' , \rho ')$ can intersect at most one ball from
  $B$. 
  Now, if $\zeta_1, \ldots , \zeta_e$ are the solutions to $f$, then we must have that some $\hat{\zeta_{i_j}}$ is
  associated to each $\zeta_j$ for some indices $1 \le i_1 < i_2 < \cdots < i_e \le d$.
  Thus, if $T$ has $e$ elements, then the only solutions to $g$ associated to $T$ are also solutions to $f$.
\end{proof}

\begin{remark}
  If $g'$ is a general square subsystem of $f$, then it will have $d$ solutions and the only common solutions to $g$ and to
  $g'$ are solutions to $f$.
  In this case, if Algorithm~\ref{alg:set} returns FAIL, then $\#T>e$, so that some pair
  of balls in Step 6 meet, but their intersection does not contain a common solution to $g$ and to $g'$.
  In this case, we may then further refine the solutions in $S,S'$, and the corresponding balls until no such extraneous
  pair of balls meet.  
\end{remark}




\subsection{Newton-Okounkov bodies and Khovanskii bases}\label{sec:nobody}
Perhaps the main difficulty in applying Algorithm~\ref{alg:ind} is obtaining the correct number $d$ beforehand.
As noted in the beginning of this section, Gr\"{o}bner bases give a general recipe for calculating this number.
Here, we sketch a less well-developed approach in the case of a square subsystem defined as in~\eqref{eq:mat}.
In this case, $d$ is given by a \emph{birationally-invariant intersection}
index over $\CC^n.$
We summarize the basic tenets of this theory as developed in~\cite{KK10,KK12}.

\begin{defn}(\cite[Def.~4.5]{KK12})
  Let $X$ be an $n$-dimensional irreducible variety over $\CC$ with singular locus $X_{sing}$.
  For an $n$-tuple  $(L_1, L_2, \dotsc , L_n)$ of finite-dimensional complex subspaces of the function field $\CC (X)$, let
  $\bs{L} = L_1 \times L_2 \times \cdots \times L_n$, and define 
\[
   U_{\textbf{L}}\ :=\ \{ z\in X \setminus X_{sing} \mid  L_i \subset \calO_{X,z} \mbox{ for }i=1,\dotsc,n \}\,,
\]
the set of smooth points where every function in each subspace $L_i$ is regular, and  
\[
   Z_{\textbf{L}}\ :=\  \bigcup_{i=1}^n \{ z\in U_{\bs{L}} \mid f(z) =0 \ \ \forall  f\in L_i\, \}\,,
\]
the set of basepoints of $\bs{L}$. 
For generic $g=(g_1,\dotsc , g_n)\in \bs{L}$, all solutions to the system $g_1(z)=\cdots = g_n(z)=0$ on
$U_{\textbf{L}} \smallsetminus Z_{\textbf{L}}$ are nonsingular and their number is independent of the choice of $g$.
The common number is the birationally invariant intersection index $[L_1,L_2,\ldots , L_n]$. 
\end{defn}

These claims are proven in~\cite[Sections 4 \& 5]{KK10}.
For our purposes, $X=\CC^n$ and $\bs{L}=L\times \cdots \times L$ where $L\subset \CC [z_1,\ldots , z_n]$
is the linear space spanned by the polynomials in our system $f$.
Write $d_L$ for this self-intersection index, note that $U_{\bs{L}} = \CC^n$, while $Z_{\bs{L}}=\variety (f)$.
Thus~\eqref{eq:bound} holds for general square subsystems of $f$, taking $d=d_L$.\medskip 

Let  $\nu \colon \CC (X)^\times \to (\ZZ^n , \prec )$ be a surjective valuation where $\prec $ is some fixed total order on
$\ZZ^n$.
For example, $\nu$ could restrict to the exponent of the leading monomial in a term order $\prec$ on $\CC[x_1,\dotsc,x_n]$.
We attach to $(L,\nu )$ the following data:
\begin{itemize}
\setlength\itemsep{0.3em}
\item $A_L=\displaystyle\bigoplus_{k=0}^\infty t^k L^k$---a graded subalgebra of $\CC (X)[t].$ 
\item $S(A_L,\nu ) = \{ (\nu (f), k) \mid f \in L^k \text{ for some } k \in \NN \}$, a sub-monoid of $\ZZ^n \oplus \NN$
  associated to the pair $(L,\nu ),$ where $L^k$ is the $\CC$-span of $k$-fold products from $L$.
  This is the \demph{initial algebra} of $A_L$ with respect to the extended valuation $\nu_t \colon\CC (X)(t)^\times \to
  (\ZZ^n \oplus \ZZ ,\prec_t)$ defined by $\nu_t (f_k\,t^k +\cdots  + f_0 ) \mapsto \left(\nu (f_k) , k \right),$ where
  $\prec_t$ is the \emph{levelwise order} defined by 
\[
    (\alpha_1 , k_1)\ \prec_t\ (\alpha_2, k_2) \quad \text{ if }
    \quad  k_1\ >\ k_2 \quad \text{ or }
    \quad k_2\ =\ k_1 \  \text{and} \quad \alpha_1\ \prec\ \alpha_2\,.
\]
\item $\ind (A_L , \nu )$---the index of $\ZZ \, S(A_L,\nu ) \cap \left( \ZZ^n \times \{ 0\} \right)$ as a subgroup of $\ZZ^n \times \{ 0 \}$.
\item $\overline{\Cone (A_L,\nu )}$---the Euclidean closure of all $\RR_{\ge 0}$-linear combinations from $S(A_L, \nu ).$
\item $\Delta (A_L,\nu )= \overline{\Cone (A_L,\nu )}\cap (\RR^{n}\times\{1\})$---the \demph{Newton-Okounkov body.}
\end{itemize}\smallskip

The linear space $L$ induces a rational \demph{Kodaira map}
\[
  \Psi_L\ \colon\ X\ --\to\ \PP (L^*)\qquad
  z\ \mapsto\ \big[ f \mapsto f(z) \big]\,,
\]
with the section ring $A_L$ the projective coordinate ring of the image.

\begin{prop}[{{\cite[Thm.~4.9]{KK12}}}]
\label{prop:kk}
  Let $L$ be a finite-dimensional subspace of $\CC (X)$.
  Then
\[
  d_L\ =\ \displaystyle\frac{n! \, \deg \Psi_L}{\ind (A_L,\nu )} \cdot \Vol \, \Delta (A_L, \nu).
\]
Here, $\Vol$ denotes the $n$-dimensional Euclidean volume in the slice $\RR^n \times \{1 \}$.
\end{prop}

In our setting, where $X=\CC^n$ and $L= \spa_\CC \{ f_1,\ldots , f_N \},$ the Kodaira map $\Psi_L $ is 
$z \mapsto [f_1(z) : f_2(z) : \cdots : f_N (z)]$.
Thus, if need be, $\deg \Psi_L$ may be computed symbolically.
The main difficulty in applying Proposition~\ref{prop:kk} is that it may be hard to determine the 
Newton-Okounkov body, as the monoid $S(A_L,\nu )$ need not be finitely generated.
This leads us to the notion of a finite Khovanskii basis~\cite{KM16}.

\begin{defn}
  \label{def:khovanskii}
  A \emph{Khovanskii basis} for $(L,\nu)$ is a set $\{a_i\mid i\in I\}$ of generators for the algebra $A_L$ whose values
  $\{\nu_t (a_i)\mid i\in I\}$ generate the monoid $S(A_L, \nu )$.
  If $<$ is a global monomial order on $k[z_1,\ldots , z_n],$ taking lead monomials defines a valuation
  $\nu\colon k[z_1,\ldots , z_n] \to (\ZZ^n,\prec)$, where $\prec $ is the reverse of $<$.
  A Khovanskii basis with respect this valuation is commonly known as a \emph{SAGBI basis}~\cite{KM89,RS90}. 
\end{defn}

When the monoid  $S(A_L,\nu )$ is finitely generated, there is a finite Khovanskii basis for $(L,\nu)$. When this occurs, we may compute the Khovanskii basis via a binomial-lifting/subduction algorithm such as described
in~\cite{RS90} or~\cite[Ch.~11]{Stu96}.

\begin{example}
We consider an ``illustrative example'' of an overdetermined system from~\cite{AHS18}: $$
\left(
\begin{array}{c}
f_1(z_1,z_2,z_3)\\
f_2(z_1,z_2,z_3)\\
f_3(z_1,z_2,z_3)\\
f_4(z_1,z_2,z_3)
\end{array}
\right)
= \left(
\begin{array}{c}
{z}_{1}^{2}+{z}_{2}^{2}-1,\,\\
-{16}\,{z}_{2}^{2}+{8}\,{z}_{1}+{17
      },\,\\
-{z}_{2}^{2}+{z}_{1}-{z}_{3}-1,\,\\
{64}\,{z}_{1}{z}_{2}+{16}\,{z}_{2}
      \end{array}
            \right)
      $$
The square subsystem defiened by $f_1=f_2=f_3=0$ has two singular solutions, and $f_4$ is the Jacobian determinant of this subsystem. Let $<$ be the graded reverse lexicographic ordering with $z_1>z_2>z_3$ and $L= \spa_\CC \{ f_1, f_2, f_3, f_4 \}.$ We observe that the initial terms of $tf_1, \ldots, tf_4 \in A_L$ under the induced order $<_t$ are given by $t\,{z}_{1}^{2},\,{-{16}\,t\,{z}_{2}^{2}},\,{-t\,{z}_{2}^{2}},$ and $64 t\,{z}_{1}{z}_{2}.$ The lattice points corresponding to these monomials comprise the first level $S(A_L, \nu ) \cap (x_4 =1)$ and lie in the linear subspace of $\RR^3 \times \{1 \} $ defined by  $x_3=0.$ We see that the inner approximation to the Newton-Okounkov body $\Delta (A_L, \nu )$ given by the first level has $3$-dimensional volume $0.$ However, there exists $x\in S( A_L, \nu )$ with $x_3\ne 0$:
      \begin{eqnarray*}
512\,t^{2}{z}_{1}^{2}{z}_{3}+6656\,t^{2}{z}_{1}{z}_{3}-6400\,t^{2}{z}_{3
      }^{2}+14000\,t^{2}{z}_{1}-26368\,t^{2}{z}_{3}-27125\,t^{2} \\
      = t^2 \, (64\,{f}_{1}{f}_{2}-21\,{f}_{2}^{2}-512\,{f}_{1}{f}_{3}+768\,{f}_{2}{f}_{
      3}-6400\,{f}_{3}^{2}+\tfrac{1}{8}\,{f}_{4}^{2}) \in A_L,
\end{eqnarray*}
giving
$$x:=\left(\begin{array}{c} 2\\ 0\\ 1 \\ 2 \end{array} \right) \in S(A_L).$$ This element of $A_L$ was obtained by the previously mentioned binomial-lifting/subduction algorithm---carrying this out further, we can verify that this new element together with the original generators give a finite Khovanskii basis for $A_L.$ It follows that
$$\Delta (A_L, <_t) = \conv
\Bigg(
\left( \begin{array}{c} 1 \\ 0 \\ 0 \\1\end{array} \right),
\left( \begin{array}{c} 0 \\ 2 \\ 0 \\1\end{array} \right),
\left( \begin{array}{c} 1 \\ 1 \\ 0 \\1\end{array} \right),
\left( \begin{array}{c} 2 \\ 0 \\ 0 \\1\end{array} \right),
\left( \begin{array}{c} 1 \\ 0 \\ 1/2 \\1\end{array} \right)
\Bigg)
$$
and $\Vol (A_L) = 1.$ We also have that $\deg \Psi_L = 2$ and $\ind (A_L)=1.$ Thus, we have $d_L=2,$ giving a total root count of $4$ after squaring up $f.$
\end{example}

\begin{remark}
We note that the total root count for a system $g=(g_1,g_2,g_3)$ obtained by randomizing $f$ in the previous example is
equal to the normalized volume of the common Newton polytope of $g_1,g_2,g_3$, which is four. 
This is equal to the ``expected'' polyhedral root count~\cite{Kus76, Ber75}.
However, this information does not give us $d_L=2$ as needed for Algorithms~\ref{alg:ind} and~\ref{alg:set}.
\end{remark}

\section{Certification via liaison pruning}
\label{sec:liaison}
Suppose that we have an overdetermined system $f$ with a square subsystem $g$, so that $\variety(f)\subset\variety(g)$.
Suppose further that we have a square system $h$ with $\variety(h)=\variety(g)\smallsetminus\variety(f)$.
Given this, we may certify all approximate solutions to $g$ and then certify the subset of those
that are approximate solutions to $h$, so that the solutions in $\variety(g)\smallsetminus\variety(h)$ which remain
are certifiably approximate solutions to $f$.
This solves Problem 1.
When this occurs, we say that $\variety(f)$ is in liaison with the complete intersection $\variety(h)$. The basic scheme for certification via liaison is Algorithm~\ref{alg:liaisonI}. We also give a generalized version, Algorithm~\ref{alg:liaisonII}, which we later apply to the Schubert calculus in Section~\ref{subsec:schubertExample}.

Let us begin with some definitions.
A system $g_1,\dotsc,g_r$ of $r$ polynomials is a \demph{complete intersection} if the variety
$\variety(g_1,\dotsc,g_r)\subset\CC^n$ they define has dimension $n-r$, equivalently if it has \demph{codimension} $r$. 
A square system is a zero-dimensional complete intersection.

More generally, varieties $X,Y\subset\CC^n$ of codimension $r$ are in \demph{liaison} if there are polynomials
$g_1,\dotsc,g_r\in\CC[x_1,\dotsc,x_n]$ such that $\variety(g_1,\dotsc,g_r)=X\cup Y$.
This relation has been deeply studied (see~\cite{KMMNP} and the references therein).  
Of particular interest is when one of the varieties, say $Y$, is itself a (different) complete intersection, so that $X$ is
in liaison with a complete intersection.
(This is a special case of the {\it licci} equivalence relation.)

\begin{example}\label{Ex:RatNormCurve}
 {\bf (The twisted cubic.)}
 The closure of the set $\{[1,t,t^2,t^3]\mid t\in\CC\}$ is the rational normal curve $C\subset\PP^3$.
 It is defined by three quadrics, $wy-x^2, wz-xy, xz-y^2$, and is thus not a complete intersection.
 In the affine patch $\CC^3$ defined by $w=1$, if we use the difference of the first two generators
 and the last generator, 
 then $\variety(z-y+x^2-xy,xz-y^2)=C\cup\ell$, where
 $\ell=\variety(x-y,x-z)$ is the line $\{(t,t,t)\mid t\in\CC\}$. 
 \[
   \begin{picture}(100,100)
      \put(0,0){\includegraphics{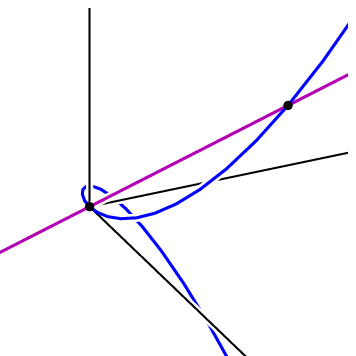}}
      \put(69,5){$x$} \put(93,50){$y$} \put(17.5,93){$z$}
      \put(3,35){$\ell$}  \put(47,32){$C$}
    \end{picture}
 \]
\end{example}

Let $X\subset\CC^n$ be a variety of codimension $r$ that is in liaison with a complete intersection $Y$.
There are  polynomials $f=(f_1,\dotsc,f_s)$, $g=(g_1,\dotsc,g_r)$, and $h=(h_1,\dotsc,h_r)$ such that
\[
   X\ =\ \variety(f)\,,\qquad
   X\cup Y\ =\ \variety(g)\,,\quad\mbox{ and }\quad
   Y\ =\ \variety(h)\,.
\]
We generalize the notion of a square system of polynomials. A \demph{square system} on $X$ consists of polynomials $g_{r+1},\dotsc,g_n$ that are sufficiently general in that
$X\cap\variety(g_{r+1},\dotsc,g_n)$ is a finite set and the intersection is transverse.
Then
 \begin{equation}\label{Eq:liaison_reveal}
   \variety(g_1,\dotsc,g_n)\ =\
    \bigl(X\cap\variety(g_{r+1},\dotsc,g_n)\bigr)\:\cup\:
    \bigl(Y\cap\variety(g_{r+1},\dotsc,g_n)\bigr)\:.
 \end{equation}
Thus the square system $X\cap\variety(g_{r+1},\dotsc,g_n)$ on $X$ is the set-theoretic difference of two square systems of
polynomials, $\variety(g_1,\dotsc,g_n)$ and
 \begin{equation}\label{Eq:liaison_square}
  \variety(h_1,\dotsc,h_r\,,\,g_{r+1},\dotsc,g_n)\ =\ Y\cap\variety(g_{r+1},\dotsc,g_n)\,.
 \end{equation}
For example, let $C$ be the rational normal curve of Example~\ref{Ex:RatNormCurve} in $\CC^3$, which has codimension 2, so
that $C\cap \variety(x+y+z+1)$ is a square system on $C$.
Manipulating the polynomials in $\variety(z-y+x^2-xy,xz-y^2,x+y+z+1)$ leads to the solutions
\[
  (-\tfrac{1}{3},-\tfrac{1}{3},-\tfrac{1}{3})\ \mbox{ on $\ell$ \ and \ }
  (-1,1,-1)\mbox{ and }(\pm\sqrt{-1},-1,\mp\sqrt{-1})\ \mbox{ on $C$\,.}
\]
If $f_1,\dotsc,f_s$ generate the ideal of $X$, then $X\cap\variety(g_{r+1},\dotsc,g_n)$ is the overdetermined system
\[
  f \ =\ (f_1,\dotsc,f_s,g_{r+1},\dotsc,g_n)
\]
Thus an algorithm to certify points on $X\cap\variety(g_{r+1},\dotsc,g_n)$ solves Problem 1 for $f$.
As we may certifty solutions and nonsolutions to systems~\eqref{Eq:liaison_reveal} and~\eqref{Eq:liaison_square}, this
discussion leads to the following certification algorithm,
when a variety $X$ is in liaison with a complete intersection $Y$.
This uses the test of Proposition~\ref{prop:reject}, the Taylor
residual~\eqref{eq:taylor_bound}, and Smale's $\alpha$-theory for the system~\eqref{Eq:liaison_square}.

\begin{algorithm}[Certifying approximate solutions to a square system on a variety $X$]\label{alg:liaisonI}$ $\\
  \textbf{Input:} $(r, g, h, S )$
  \begin{itemize}
  \item[] $r \in \NN$
  \item[] $g=(g_1,\dotsc,g_n)$ --- a square polynomial system such that $\variety(g_1,\dotsc,g_r)=X\cup Y$,
      \mbox{\qquad}with both $X$ and $Y$ of codimension $r$
  \item[] $h=(h_1,\dotsc,h_r)$ --- polynomials such that $\variety(h)=Y$
  \item[] $S = \{ \hat{\zeta_1 }, \ldots , \hat{\zeta_m} \}$ --- pairwise distinct
    approximate solutions to $g$  with refinement\newline \mbox{\qquad} operator $\calN_g$
  \end{itemize}
  \textbf{Output:} $T,U\subset S$ with $S=T\sqcup U$, where $T$ consists of approximate solutions to
  $X\cap\variety(g_{r+1},\dotsc,g_n)$ and $U$ consists of approximate solutions to
  $Y\cap\variety(g_{r+1},\dotsc,g_n)$.\vspace{-10pt}
  \begin{algorithmic}[1]
    \STATE  Set $\defcolor{f}:=(h_1,\dotsc,h_r\,,\,g_{r+1},\dotsc,g_n)$, a square system on $Y$.
    \STATE Initialize $T \gets \emptyset $, $U \gets \emptyset $
    \FOR{ \indent $\hat{\zeta}\in S$ }
     \STATE{$\zeta' \gets \hat{\zeta}$}
     \STATE{\textbf{ if } $\alpha(f, \zeta') < \frac{13-3\sqrt{17}}{4}$
            \textbf{ then } $U\gets U \cup \{ \hat{\zeta} \}$}
     \STATE{\textbf{ \  else if } $\delta(f,g, \zeta') > 0 $
            \textbf{ then } $T\gets T \cup \{ \hat{\zeta} \}$}
     \STATE{\textbf{ \ else } $\zeta' \gets \calN_g(\zeta')$ and return to 5.}
     \STATE{\textbf{ end if }}
       \ENDFOR
  \end{algorithmic}
\end{algorithm}

\begin{remark}
\label{remark:eff+alpha}
As in all subsequent algorithms, we assume distinct approximate solutions to $g$ with
refinement operator $\calN_g$ as part of the input.
We could have just as easily assumed effective approximate solutions.
The test in line 5 could be replaced by testing that $\zeta '$ is an approximate solution to the square system $f$ by some criterion other than $\alpha$-theory---for simplicity, we do not assume this criterion is part of the input.
\end{remark}

\begin{proof}[Proof of correctness]
 As $\hat{\zeta}\in S$, it is an approximate solution to the square system $g$ with an associated nonsingular solution
 $\zeta\in\variety(g)\subset X\cup Y$.
 Since $\zeta$ is nonsingular, $\zeta\not\in X\cap Y$, as $X\cup Y$ is singular along $X\cap Y$.
 Thus $\zeta\in X$ if and only if $\zeta\not\in Y$.
 Let $\{\hat{\zeta}_i\mid i\in\NN\}$  be the sequence of iterates using $\calN_g$ starting at $\hat{\zeta}$.
 This converges to $\zeta$.
  
 If $\zeta\in Y$, then $\zeta\in\variety(f)$, and the sequence $\{\hat{\zeta}_i\}$  will eventually lie in the basin of
 quadratic convergence for Newton iterations $N_f$ and $\beta(f, \hat{\zeta}_i)$ converges to $0$.
 As $\gamma(f, \hat{\zeta}_i)$ is bounded, $\alpha(f, \hat{\zeta}_i)=\gamma(f, \hat{\zeta}_i)\cdot\beta(f,\hat{\zeta}_i)$
 converges to $0$.
 Thus the condition in line 5 will eventually hold and $\hat{\zeta}$ will be placed in $U$.

 If $\zeta\not\in Y$, then $\zeta\not\in\variety(f)$.
 By Corollary~\ref{cor:refine_and_reject}, the Taylor residuals $\delta(f,g,\zeta_j)$ are positive for $j$ large enough.
 Thus the condition in line 6 eventually holds, and $\hat{\zeta}$ will be placed into $T$.
\end{proof}

We describe a more involved application of this idea.
Write  $\defcolor{\codim X}$ for the codimension, $n{-}\dim X$, of a variety $X\subset\CC^n$.
Suppose that $X_1,\dotsc,X_m\subset\CC^n$ are in general position and
$\sum \codim X_i=n$, then Bertini's Theorem~\cite{Kl74} implies that 
 \begin{equation}\label{Eq:intersection}
   X_1\bigcap X_2 \bigcap \dotsb \bigcap X_m
 \end{equation}
is a transverse intersection consisting of finitely many points.
When $n=m$, so that each $X_i= \variety (f_i) $ is a hypersurface, then~\eqref{Eq:intersection} is equivalent to the square
polynomial system
\[
  f_1\ =\ f_2\ =\ \dotsb\ =\ f_n\ =\ 0\,.
\]
As a variety need not be a complete intersection, a \demph{square system of varieties}~\eqref{Eq:intersection}  with $m<n$
does not necessarily have a formulation as a square system of polynomials. However, the points of ~\eqref{Eq:intersection} are the solutions to an overdetermined system of polynomials given by
  the generators of the ideals of each of $X_1,\dotsc,X_m$.

Suppose now that $X_1,\dotsc,X_m\subset\CC^n$ form a square system of varieties~\eqref{Eq:intersection}, each 
$X_i$ is in liaison with a complete intersection $Y_i$, and these are all in sufficiently general position.
Then there are square systems $g_1,\dotsc,g_n$ and $h_1,\dotsc,h_n$ of polynomials such that if
$\defcolor{a_\bullet}\colon 0=a_0<a_1<\dotsb<a_m=n$ is defined by $a_i-a_{i-1}=\codim X_i (=\codim Y_i)$ for each $i$, then
 \begin{equation}\label{Eq:lcli}
  \variety(g_{1+a_{i-1}},\dotsc,g_{a_i})\ =\ X_i\,\cup\, Y_i
    \qquad\mbox{and}\qquad
    \variety(h_{1+a_{i-1}},\dotsc,h_{a_i})\ =\  Y_i\,,
  \end{equation}
are complete intersections for each $i=1,\dotsc,m$.
Thus
 \begin{equation}\label{Eq:ManyExcess}
   \variety(g)\ =\ \bigcap_{i=1}^m \bigl( X_i \cup Y_i\bigr)\ .
 \end{equation}
We give a more general version of Algorithm~\ref{alg:liaisonI} that will certify solutions to the square
system~\eqref{Eq:intersection} of varieties, given solutions~\eqref{Eq:ManyExcess} to the square system $g$.

\begin{algorithm}[Certifying solutions to a square system of varieties]\label{alg:liaisonII}$ $\\
  \textbf{Input:} $(a_\bullet, g, h, S )$
  \begin{itemize}
  \item[] $a_\bullet\colon 0=a_0<a_1<\dotsb<a_m=n$
  \item[] $g=(g_1,\dotsc,g_n)$ and $h=(h_1,\dotsc,h_n)$ --- square polynomial systems such that\newline
        \mbox{\qquad} for each $i=1,\dotsc,m$,~\eqref{Eq:lcli} are complete intersections.
  \item[] $S = \{ \hat{\zeta_1 }, \ldots , \hat{\zeta_s} \}$ --- pairwise distinct approximate solutions to $g$ 
  \end{itemize}
  \textbf{Output:} $T \subset S$ consisting of approximate solutions to $X_1\cap X_2\cap\dotsb\cap X_m$.
  \begin{algorithmic}[1]
  \FOR{ \indent $i=1,\dotsc, m$ }
    \STATE  Set $f:=(g_1,\dotsc,g_{a_{i-1}}\,,\,h_{1+a_{i-1}},\dotsc,h_{a_i}\,,\,g_{1+a_i},\dotsc,g_n)$.
    \STATE Initialize $T \gets \emptyset $
    \FOR{ \indent $\hat{\zeta}\in S$ }
     \STATE{$\zeta' \gets \hat{\zeta}$}
     \STATE{\textbf{ if } $\alpha(f, \zeta') < \frac{13-3\sqrt{17}}{4}$
            \textbf{ then } discard $\hat{\zeta}$}
     \STATE{\textbf{ \  else if } $\delta(f,g, \zeta') > 0 $
            \textbf{ then } $T\gets T \cup \{ \hat{\zeta} \}$}
     \STATE{\textbf{ \ else } $\zeta' \gets \mathcal{N}_g(\zeta')$ and return to 6.}
     \STATE{\textbf{ end if }}
     \ENDFOR
     \STATE $S\gets T$
  \ENDFOR
  \end{algorithmic}
\end{algorithm}

\begin{proof}[Proof of correctness]
  By algorithm~\ref{alg:liaisonI}, in each iteration $i=1,\dotsc,m$ of the outer loop, the algorithm constructs the 
  set $T$ of elements of the input $S$ that do not lie in $Y_1\cup\dotsb\cup Y_i$.
  As $S\cap X_1\cap\dotsb\cap X_m=S\smallsetminus(Y_1\cup\dotsb\cup Y_m)$, we see that the algorithm performs as claimed.
\end{proof}


\section{Examples}
\label{sec:examples}

\label{sec:quartics}We give three further examples that illustrate our certification via square subsystems.
All computations were carried out using the computer algebra system Macaulay2~\cite{M2}.
For each example, we found complex floating-point solutions to square subsystems via
homotopy continuation, as implemented in the package
\texttt{NumericalAlgebraicGeometry}~\cite{NAG4M2}.
Tests from $\alpha$-theory were supplied by the
package~\texttt{NumericalCertification}~\cite{NSLee}.

\subsection{Plane quartics through four points}
\label{subsec:quarticsExample}
Consider the overdetermined system $f=(f_1,\dotsc,f_{11}),$ where the the $f_i$ are given as follows:
\begin{eqnarray*}
  &z_1z_2-z_2^2+z_1-z_2\,,\ z_1^2-z_2^2+4z_1-4z_2\,,\ z_2^3-6z_2^2+5z_2+12\,,&\\ 
  & z_1z_2^2-6z_2^2-z_1+6z_2+12\,,\ z_1^2z_2-6z_2^2-4z_1+9z_2+12\,,\ z_1^3-6z_2^2-13z_1+18z_2+12\,,&\\ 
  &z_2^4-31z_2^2+42z_2+72\,,\ z_1z_2^3-31z_2^2+z_1+41z_2+72\,,\ z_1^2z_2^2-31z_2^2+4z_1+38z_2+72\,,&\\ 
  & z_1^3z_2-31z_2^2+13z_1+29z_2+72\,,\ z_1^4-31z_2^2+40z_1+2z_2+72\,.&
\end{eqnarray*}
These give a basis for the space of quartics passing through the four points:
\[
(4,4)\,,\ (-3,-1)\,,\ (-1,-1)\,,\ (3,3)\ \in\  \CC^2.
\]
As an illustration of Algorithm~\ref{alg:ind} and the techniques based on Khovanskii bases described in Section~\ref{sec:nobody}, we show how to certify that numerical approximations
of these points represent true solutions to $f$.

Letting $L=\spa_{\CC} \{ f_1,\ldots , f_{11} \},$ we consider the algebra $A_L.$ Letting $<$ be the graded-reverse lex
order with $z_1 > z_2,$ the algebra $A_L$ has a finite Khovanskii basis with respect to the $\ZZ^2$-valuation associated to $<$.
It is given by $S=\{ t \, f_1, t \, f_2 , \ldots , t \, f_{11},  t^2\,g, t^3\, h \},$ where
\begin{align*}
     g &= z_1 \, z_2^3-z_2^4+10 \, z_1^2 \, z_2-26 \, z_1 \, z_2^2 +16 \, z_2^3+10 \, z_1^2-15 \, z_1 \, z_2
           +5 \, z_2^2+12 \, z_1 -12 \, z_2 
      \\
    h &= 10 \, z_1^4 \, z_2-49 \, z_1^3 \, z_2^2+89 \, z_1^2 \, z_2^3-71 \, z_1 \, z_2^4+21 \, z_2^5
          +10 \, z_1^4-18 \, z_1^3 \, z_2-18 \,  z_1^2 \, z_2^2\\
     \, &+ 50 \, z_1 \, z_2^3-24 \, z_2^4+31 \, z_1^3-83 \, z_1^2 \, z_2
        +73 \, z_1 \, z_2^2-21 \, z_2^3+24 \, z_1^2-48 \,  z_1 \, z_2+24 \, z_2^2.
\end{align*}

The Newton-Okounkov body, depicted below, has normalized volume $12$. The integer points correspond to $f_1,\dotsc,f_{11}$. The fractional vertices corresponding to $t^2 g$ and $t^3 h$ demonstrate that these elements are essential in forming the Khovanskii basis.

\begin{center}
\includegraphics{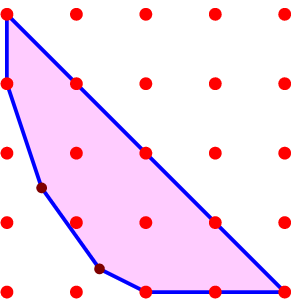}
\end{center}

Using the procedure of~\cite{SS88}, we may express $g$ and $h$ as homogeneous polynomials in the algebra generators $f_1, \ldots, f_{11}$:

\begin{flalign*}
g &= -\displaystyle\tfrac{5452243}{3803436} {f}_{4}
       {f}_{9}+\displaystyle\tfrac{1088119}{7606872} {f}_{5}
       {f}_{9}-\displaystyle\tfrac{179087}{7606872} {f}_{6}
       {f}_{9}-\displaystyle\tfrac{1184975}{7606872} {f}_{8}
       {f}_{9}+\displaystyle\tfrac{2728589}{7606872}
       {f}_{9}^{2}-\displaystyle\tfrac{5046}{3913} {f}_{1}
       {f}_{10}\\
       &+\displaystyle\tfrac{5951}{11739} {f}_{2}
       {f}_{10}+\displaystyle\tfrac{5452243}{3803436} {f}_{3}
       {f}_{10}+\displaystyle\tfrac{2196073}{1901718} {f}_{4}
       {f}_{10}+\displaystyle\tfrac{905753}{3803436} {f}_{5}
       {f}_{10}-\displaystyle\tfrac{129295}{1901718} {f}_{6}
       {f}_{10}+\displaystyle\tfrac{1184975}{7606872} {f}_{7}{f}_{10}\\
       &-\displaystyle\tfrac{5983}{29484} {f}_{8}
       {f}_{10}-\displaystyle\tfrac{2728589}{7606872} {f}_{9}
       {f}_{10}-\displaystyle\tfrac{65165}{1267812}
       {f}_{10}^{2}-\displaystyle\tfrac{5951}{11739} {f}_{1}
       {f}_{11}-\displaystyle\tfrac{9872411}{7606872} {f}_{3}
       {f}_{11}-\displaystyle\tfrac{1632419}{7606872} {f}_{4} {f}_{11}\\
       &+\displaystyle\tfrac{129295}{1901718} {f}_{5}
       {f}_{11}-\displaystyle\tfrac{1184975}{7606872} {f}_{7}
       {f}_{11}+\displaystyle\tfrac{2728589}{7606872} {f}_{8}
       {f}_{11}+\displaystyle\tfrac{65165}{1267812} {f}_{9} {f}_{11}.
  \end{flalign*} 
\small 
\begin{flalign*}
        h &= \displaystyle\tfrac{423458528993}{35955045627228} {f}_{5}
       {f}_{9}^{2}
       -\displaystyle\tfrac{348294358499}{77902598858994} {f}_{6}
       {f}_{9}^{2}
       +\displaystyle\tfrac{33023933703287}{1012733785166922} {f}_{7}       {f}_{9}^{2}
-\displaystyle\tfrac{82250093861471}{6076402711001532} {f}_{8}
       {f}_{9}^{2}\\
       &-\displaystyle\tfrac{4432317106115}{233707796576982}
       {f}_{9}^{3}
       -\displaystyle\tfrac{33023933703287}{1012733785166922} {f}_{7}       {f}_{8} {f}_{10}
       +\displaystyle\tfrac{1065288183977}{37508658709886} {f}_{3} {f}_{9} {f}_{10}\\
&-\displaystyle\tfrac{96715490949542}{
       317951304645429} {f}_{4} {f}_{9}
       {f}_{10}
       -\displaystyle\tfrac{49052367589004489}{1695316356369427428}
       {f}_{5} {f}_{9} {f}_{10}\displaystyle\tfrac{25940308080550879}{
       1695316356369427428} {f}_{6} {f}_{9}
       {f}_{10}\\
       &-\displaystyle\tfrac{8914885258327}{467415593153964} {f}_{7}
       {f}_{9} {f}_{10}
+\displaystyle\tfrac{848951864573779}{13671906099753447}
       {f}_{8} {f}_{9} {f}_{10}\displaystyle\tfrac{47174423433062585}{
       1695316356369427428} {f}_{9}^{2}
       {f}_{10}\\
       &-\displaystyle\tfrac{13418439080090}{56262988064829} {f}_{1}
       {f}_{10}^{2}
       +\displaystyle\tfrac{1014252370876}{12983766476499} {f}_{2}
       {f}_{10}^{2}+\displaystyle\tfrac{490676889497623}{2103370169192838}
       {f}_{3} {f}_{10}^{2}\\
       &+\displaystyle\tfrac{135026148156913879}{
       423829089092356857} {f}_{4}
       {f}_{10}^{2}
       +\displaystyle\tfrac{1125534856927697}{27343812199506894}
       {f}_{5} {f}_{10}^{2}\displaystyle\tfrac{7907079377499775}{
       423829089092356857} {f}_{6}
       {f}_{10}^{2}\\
       &+\displaystyle\tfrac{164898009266531}{54687624399013788}
       {f}_{7} {f}_{10}^{2}-\displaystyle\tfrac{5432718489778696}{
       141276363030785619} {f}_{8}
       {f}_{10}^{2}-\displaystyle\tfrac{16931183230705166}{423829089092356857}
       {f}_{9} {f}_{10}^{2}\\
        &-\displaystyle\tfrac{2073065531395802}{
       141276363030785619} {f}_{10}^{3}-\displaystyle\tfrac{1065288183977}{
       37508658709886} {f}_{3} {f}_{8}
       {f}_{11}\displaystyle\tfrac{33023933703287}{1012733785166922} {f}_{7}
       {f}_{8} {f}_{11}\\
       &+\displaystyle\tfrac{13306110674011}{225051952259316}
       {f}_{3} {f}_{9} {f}_{11}+\displaystyle\tfrac{2828825493124010}{
       13671906099753447} {f}_{4} {f}_{9}
       {f}_{11}-\displaystyle\tfrac{4516368725120116}{423829089092356857}
       {f}_{5} {f}_{9} {f}_{11}\\
       &-\displaystyle\tfrac{3780890220862891}{
       1695316356369427428} {f}_{6} {f}_{9}
       {f}_{11}-\displaystyle\tfrac{280393696081193}{6076402711001532} {f}_{7}
       {f}_{9} {f}_{11}
       -\displaystyle\tfrac{313094593927918}{13671906099753447}
       {f}_{8} {f}_{9} {f}_{11}\\
       &+\displaystyle\tfrac{5042023611256019}{
       847658178184713714} {f}_{9}^{2}
       {f}_{11}
       +\displaystyle\tfrac{14158876247624}{168788964194487} {f}_{1}
       {f}_{10} {f}_{11}-\displaystyle\tfrac{8754568627342}{168788964194487}
       {f}_{2} {f}_{10} {f}_{11}\\
       &-\displaystyle\tfrac{26911193688915259}{
       54687624399013788} {f}_{3} {f}_{10}
       {f}_{11}-\displaystyle\tfrac{4282639294736275}{18229208133004596} {f}_{4}
       {f}_{10} {f}_{11}-\displaystyle\tfrac{31468039434977}{4484963905739226}
       {f}_{5} {f}_{10} {f}_{11}\\
       &+\displaystyle\tfrac{4914325106636902}{
       423829089092356857} {f}_{6} {f}_{10}
       {f}_{11}+\displaystyle\tfrac{1833610583591729}{54687624399013788} {f}_{7}
       {f}_{10} {f}_{11}
       +\displaystyle\tfrac{16078096764903062}{
       423829089092356857} {f}_{8} {f}_{10}
       {f}_{11}\\
       &+\displaystyle\tfrac{8296532868679901}{242188050909918204}
       {f}_{9} {f}_{10} {f}_{11}+\displaystyle\tfrac{7171418221493375}{
       565105452123142476} {f}_{10}^{2}
       {f}_{11}+\displaystyle\tfrac{8754568627342}{168788964194487} {f}_{1}
       {f}_{11}^{2}\\
       &+\displaystyle\tfrac{5171410646788483}{27343812199506894}
       {f}_{3} {f}_{11}^{2}+\displaystyle\tfrac{47304126637283297}{
       1695316356369427428} {f}_{4}
       {f}_{11}^{2}
       -\displaystyle\tfrac{4914325106636902}{423829089092356857}
       {f}_{5} {f}_{11}^{2}\\
       &-\displaystyle\tfrac{107608086440381}{
       27343812199506894} {f}_{7}
       {f}_{11}^{2}-\displaystyle\tfrac{33198943704009683}{1695316356369427428}
       {f}_{8} {f}_{11}^{2}-\displaystyle\tfrac{7171418221493375}{
         565105452123142476} {f}_{9} {f}_{11}^{2}.
\end{flalign*}
\normalsize
 
The Khovanskii basis was computed using the Macaulay2 package
\texttt{SubalgebraBases}, based on the work in~\cite{ST99}.
We checked this computation against our own top-level implementation of the binomial-lifting / subduction algorithm.

\medskip

For certification, we squared up $f$ with a random matrix, $g=Af$, and found $16$ complex approximate
solutions to $g$ using homotopy continuation.
Each solution was softly certified distinct via $\alpha$-theory.
Computing values $\delta (f,g,\cdot )$ as in Algorithm~\ref{alg:ind}, we softly certified
$12$ of these as nonsolutions to $f$, hence associating the four remaining solutions to
$f$. Observe that $d_L=12 \, \deg \Psi_L$ by Proposition~\ref{prop:kk}. Also, we have $d_L\le 16$ by B\'ezout's theorem.
This implies that $\deg \Psi_L=1$ and hence $d_L=12.$

\subsection{Example from Schubert calculus}
\label{subsec:schubertExample}

We describe a family of examples from Schubert calculus to which Algorithms~\ref{alg:ind},~\ref{alg:set},
and~\ref{alg:liaisonII} all apply.
For more on the Grassmannian and Schubert calculus, see~\cite{Fulton}.
Let $m\geq 2$ be an integer and set $n:=m{+}2$.
Consider the geometric problem of the 2-planes $H$ in $\CC^n$ that meet $m$ general codimension $3$ planes nontrivially.
The number of such 2-planes is the Kostka number $K_{m^2,2^m}$, the first few values of which are shown below.
\[
  \begin{tabular}{|l||c|c|c|c|c|c|c|c|c|c|c|c|c|c|}\hline
    $m$& 1 &2 & 3 & 4 & 5 & 6 & 7 & 8 & 9 & 10 & 11 & 12 & 13 & 14 \\\hline\hline
    $K_{m^2,2^m}$&0&1 & 1 & 3 & 6 & 15 & 36 & 91 & 232 & 603 & 1585 & 4213 & 11298 & 30537 \\\hline
  \end{tabular}
\]
This may be computed recursively.
Let $\kappa_{m,i}$ be the coefficient of the Schur function $S_{(m+i,m-i)}$ in the product $(S_{(2,0)})^m$.
Then $K_{m^2,2^m}=\kappa_{m,0}$.
For the recursion, set $\kappa_{1,1}:=1$ and $\kappa_{1,0}=\kappa_{m,j}:=0$, when $j>m$.
Then, for $m>1$, we set $\kappa_{m,0}:=\kappa_{m-1,1}$ and for $j>0$,
$\kappa_{m,j}:=\kappa_{m-1,j-1}+\kappa_{m-1,j}+\kappa_{m-1,j+1}$.

We express this geometric probem in local coordinates.
Write \defcolor{$I_2$} for the $2\times 2$ identity matrix and let \defcolor{$Z$} be a $2\times m$ matrix of
indeterminates, and set $\defcolor{H}:=(Z|I_2)^\top$, which has $n$ rows and $2$ columns.
For any choice of $Z\in\mbox{Mat}_{2\times m}(\CC)$, the column span of $H$, also written $H$, is a 2-plane in
$\CC^n=\CC^m\oplus\CC^2$ that does not meet the coordinate plane $\CC^m\oplus\{0\}$, and $\mbox{Mat}_{2\times m}(\CC)$
parametrizes the set of such $2$-planes.
For $k=1,\dotsc,m$, let \defcolor{$K_k$} be a general $n\times(m{-}1)$-matrix whose column span (also written $K_k$) is a
general $(m{-}1)$-plane.
Then $\dim H\cap K_k\geq 1$ if and only if the matrix $(H|K_k)$ has rank at most $m$.
This condition is given by the $n$ maximal minors $f_{k,1},\dotsc,f_{k,n}$ of  $(H|K_k)$, each of which is the determinant
of the square $(n{-}1)\times(n{-}1)$-matrix obtained by deleting a row from  $(H|K_k)$.
This gives a system $\defcolor{f}=(f_{k,j}\mid k=1,\dotsc,m\mbox{ and }j=1,\dotsc,n)$ of $mn$ quadratic equations in $2m$
variables which define the solutions to our geometric problem. 

Any polynomial $g$ that is a linear combination of the $f_{k,j}$ has the form $g=\det(H|K_k|\ell)$, where the entries of
$\ell$ are the coefficients of $(-1)^j f_{k,j}$ in that linear combination.
This justifies the following scheme to obtain a square subsystem of $f$.
For each $k=1,\dotsc,m$ and $i=1,2$, let $L_{k,i}\supset K_k$ be an $m$-plane that is general given that it contains $K_k$.
We obtain the matrix of $L_{k,i}$ by appending a general column vector to the matrix of $K_k$.
Let $g_{k,i}$ be the determinant of the matrix $(H|L_{k,i})$---this vanishes when $\dim H\cap L_{k,i}\geq 1$.
We claim that the susbsystem $g=(g_{1,1},g_{1,2},\dotsc,g_{m,1},g_{m,2})$  of $f$ is square.

For this, let us investigate the corresponding geometric loci in the Grassmannian $G(2,n)$.
Write $\Omega_{\sT}K_k$ for the set of all 2-planes which meet $K_k$ nontrivially, and $\Omega_{\sI}L_{k,i}$ for those
that meet $L_{k,i}$ nontrivially.
Let $\Lambda_k$ be the hyperplane containing both $L_{k,1}$ and $L_{k,2}$, and let $\Omega_{\sII}\Lambda_k$ be the set of
all 2-planes that are contained in $\Lambda_k$.
Since $L_{k,1}\cap L_{k,2}=K_k$ and $L_{k,1}+ L_{k,2}=\Lambda_k$ it was shown in~\cite{Sot97} that
 \begin{equation}\label{Eq:Schubert_CI}
  \Omega_{\sI}L_{k,1} \bigcap\Omega_{\sI}L_{k,2}\ =\ \Omega_{\sT}K_k \cup \Omega_{\sII}\Lambda_k
 \end{equation}
is a (generically) transverse intersection.

It is natural to analyze this geometric problem in the context of liaison theory discussed in Section~\ref{sec:liaison}:
specifically, we can use Algorithm~\ref{alg:liaisonII}. On the other hand, the algorithms from section~\ref{sec:nonsol}
work just as well. For each approach, we explain the details needed in order to certify. We note that the main bottleneck,
solving $g,$ is well within the capabilities of modern homotopy continuation software, say, for $m$ in the single
digits~\cite{LDSVV}.\medskip

\noindent{\bf Algorithm~\ref{alg:liaisonII}.}
In the local coordinates $H=(Z|I_2)^\top$, we have that $\Omega_{\sI}L_{k,i}=\variety(g_{k,i})$, so
that~\eqref{Eq:Schubert_CI} is a complete intersection 
and $\Omega_{\sT}K_k=\variety(f_{k,1},\dotsc,f_{k,n})$  is in liaison with $\Omega_{\sII}\Lambda_k$, which we show is a
complete intersection.
Let \defcolor{$\lambda_k$} be the linear form (a row vector) whose kernel is $\Lambda_k$.
Then $H\in\Omega_{\sII}\Lambda_k$ if and only if $H\subset\Lambda_k$, so that
$\lambda_k H=\left(\begin{smallmatrix}0\\0\end{smallmatrix}\right)$.
If $h_{k,1}$ and $h_{k,2}$ are the two rows of $\lambda_k H$, then 
$\Omega_{\sII}\Lambda_k=\variety(h_{k,1},h_{k,2})$, showing that it is a complete intersection.

Our geometric problem of the 2-planes $H$ that meet each of $K_1,\dotsc,K_m$ is equivalent to the intersection
\[
   \Omega_{\sT}K_1 \;\bigcap\;
   \Omega_{\sT}K_2 \;\bigcap\;\dotsb \;\bigcap\;
   \Omega_{\sT}K_m \;,
\]
which is a square system of varieties~\eqref{Eq:intersection}.
As each is in liaison with a complete intersection, Algorithm~\ref{alg:liaisonII} applies and may
be used to certify the solutions to our geometric problem.
Its input is the set $\variety(g)$, which consists of the points in the intersection
 \begin{equation}\label{Eq:GeneralEnough}
   \Omega_{\sI}L_{1,1} \bigcap\Omega_{\sI}L_{1,2}
   \;\bigcap\; \dotsb\;\bigcap\;\Omega_{\sI}L_{k,i}\;\bigcap\; \dotsb\;\bigcap\;
   \Omega_{\sI}L_{m,1} \bigcap\Omega_{\sI}L_{m,2}\,.
 \end{equation}
While each pair $\Omega_{\sI}L_{k,1} \bigcap\Omega_{\sI}L_{k,2}$ is not in general position, this intersection is
generically transverse, and the different pairs are in general position, so the intersection~\eqref{Eq:GeneralEnough}
is transverse.
Consequently, the number of points in the intersection~\eqref{Eq:GeneralEnough} is the expected number, which is the Catalan number $C_{m}:=\frac{1}{m+1} \binom{2m}{m}$.

\noindent{\bf Algorithm~\ref{alg:ind}.}
For this algorithm, the number $d$ of excess solutions is $\frac{1}{m+1} \binom{2 m}{m}-K_{m^2,2^m}$.
It starts with the set $S=\variety(g)$ of $\frac{1}{m+1}\binom{2m}{m}$ points in the intersection~\eqref{Eq:GeneralEnough}.  
We run Algorithm~\ref{alg:ind}, and if it finds that $\#R=d$, so that we have rejected all nonsolutions, then those that
remain are certified solutions to our geometric problem $\variety(f)$.
Otherwise, we may refine the approximate solutions $\hat{\zeta}$ in $S$ so that the Newton steps $\beta(g,\hat{\zeta})$
become small enough to reject $d$ nonsolutions.

This algorithm is particularly easy in this case as the Taylor residual~\eqref{eq:taylor_bound} of a linear function $\phi$
is $|\phi(\hat{\zeta})|-\|\phi'\|\delta$, where the derivative $\phi'$ of $\phi$ is a vector.
\medskip

\noindent{\bf Algorithm~\ref{alg:set}.}
Here, we simply observe that $d=\frac{1}{m+1}\binom{2m}{m}$ is the number of solutions to the square system $g$ and $e=K_{m^2,2^m}$ is the number of
solutions to $f$. These data together with a full set of approximate solutions to $g,$ are all that is needed for certification.

\subsection{Essential matrix estimation}
\label{subsec:vision}


A fundamental object of study in geometric computer vision is the essential variety
\begin{equation}
  \label{eq:defVess}
  V_{\ess} := \{ E \in \PP (\CC^{3\times 3}) \mid E E^\top E - \tfrac{1}{2} \tr (E E^\top ) E = 0, \det E = 0\}.
\end{equation}
This is an irreducible variety of dimension $5$ and degree $10.$ Elements of $V_{\ess}$ are called
\demph{essential matrices}.
The ten polynomials defining $V_{\ess}$ are known as the \demph{Demazure cubics}~\cite{Dem88}.
They minimally generate the homogeneous ideal of $V_{\ess}$. It is possible to recover an essential matrix given five generic point \demph{correspondence constraints} of the form
\begin{equation}
    \label{eq:ptptE}
    y_i^\top E x_i = 0 \text{ for } i=1,\ldots , 5,
  \end{equation}
where $x_1,\ldots , x_5, y_1, \ldots , y_5 \in \PP (\CC^3).$ Although the overdetermined family given by equations~\eqref{eq:defVess} and~\eqref{eq:ptptE} is fairly simple, it is notable for its apperance in applications. State of the art algorithms for solving these equations run on the order of microseconds~\cite{Nis04} and have been successfully employed in large-scale 3D reconstruction pipelines~\cite{SSS08}. This motivates the problem of developing certification techniques for this problem with comparable efficiency. For concreteness, we consider an instance of this problem in which the data are given by

\begin{center}
\begin{tabular}{ccccc}
  $x_1 = \begin{pmatrix}
       {0}\\
       {0}\\
       1\end{pmatrix},\,
       $
       &
       $
       x_2 = \begin{pmatrix}
       {0}\\
       1\\
       1\end{pmatrix},\,
       $
       &
       $
       x_3 = \begin{pmatrix}
       {.750733}\\
       {.393279}\\
       1\end{pmatrix},\,
       $
       &
       $
       x_4 = \begin{pmatrix}
       {.383872}\\
       {.210436}\\
       1\end{pmatrix},\,
       $
       &
       $
       x_5 = \begin{pmatrix}
       {.970556}\\
       {.699694}\\
       1\end{pmatrix},$\\[1.85em]
$y_1 = \begin{pmatrix}
       {0}\\
       {0}\\
       1\end{pmatrix},\,
              $
       &
       $
       y_2 = \begin{pmatrix}
       {0}\\
       1\\
       1\end{pmatrix},\,
              $
       &
       $
       y_3 = \begin{pmatrix}
       {.355041}\\
       {.153766}\\
       1\end{pmatrix},\,
              $
       &
       $
       y_4 = \begin{pmatrix}
       {.090869}\\
       {.143374}\\
       1\end{pmatrix},\,
              $
       &
       $
       y_5 = \begin{pmatrix}
       {.003463}\\
       {.17189}\\
       1\end{pmatrix}.$
\end{tabular}
\end{center}

Truncating to $6$ decimal places, we may regard each of the $x_i$ and $y_i$ as rational vectors and seek a certificate. A candidate approximate solution is given to 6 places by

$$\hat{E}=
\begin{pmatrix}
      1&{-{2.36148}}&{-{.017451}}\\
      {2.52018}&{.979523}&{-{.066457}}\\
      {.117939}&{-{.913067}}&{-{10^{-6}}}\end{pmatrix}.
$$

To certify $\hat{E},$ we consider the square system $g$ given by
 equations~\eqref{eq:ptptE}, the chart $e_{1,1}=1,$ and the first three Demazure cubics: namely 
  \begin{align}
    \label{eq:dem3}
    ({e}_{2,1}^{2}+{e}_{2,2}^{2}+{e}_{2,3}^{2}+ {e}_{3,1}^{2}+{e}_{3,2}^{2}+{e}_{3,3}^{2}) e_{1,i} &+ ({e}_{1,1}{e}_{2,1}+{e}_{1,2}{e}_{2,2}+{e}_{1,3}{e}_{2,3}) e_{2,i} \\
    &+ ({e}_{1,1}{e}_{3,1}+{e}_{1,2}{e}_{3,2}+{e}_{1,3}{e}_{3,3}) e_{3,i}\nonumber 
  \end{align}
The number of solutions to $g$ is bounded \emph{a priori} by $27,$ and we can easily certify that this bound is attained. It follows that we may apply the methods described in Section~\ref{sec:nonsol}. However, there is an even simpler procedure based on the exclusion criteria of Proposition~\ref{prop:reject} and~\ref{cor:refine_and_reject}, as well as the following result obtained by symbolic computation.
  \begin{prop}
    \label{prop:VsqDecomp}
    Let $V_{sq}$ be the subvariety of $\PP (\CC^{3\times 3})$ defined by equations~\eqref{eq:ptptE} and~\eqref{eq:dem3}. Consider the polynomials 
    \begin{itemize}
    \item[] $f_1 := e_{1,1}$
    \item[] $f_2 := {e}_{1,1}{e}_{3,1}+{e}_{1,2}{e}_{3,2}+{e}_{1,3}{e}_{3,3}$
    \item[] $f_3 := {e}_{1,1}^{2}+{e}_{1,2}^{2}+{e}_{1,3}^{2}$
    \end{itemize}
    and define projective varieties as the Zariski closures of the indicated quasiprojective varieties,
    \begin{itemize}
    \item[] $V_1 := \overline{V_{sq} \setminus \variety (f_1)}$
    \item[] $V_2 :=  \overline{V_1 \setminus \variety (f_2) }$
    \item[] $V_3 := \overline{V_2 \setminus \variety (f_3)}$
\end{itemize}
We have that $V_3 = V_{\ess}.$
\end{prop}
  \begin{proof}
This follows by symbolically computing ideal quotients: letting $I_0$ by the ideal generated by~\eqref{eq:ptptE} and~\eqref{eq:dem3} and $I_k = I_{k-1}:f_k$ for $k=1,2,3,$ we may establish that $I_3$ equals the ideal defining $V_{\ess}$; for instance, by showing they have the same reduced Gr\"{o}bner basis for a given term order.  This is easily accomplished by Macaulay2.
\end{proof}

  For generic data $(x_i, y_i),$ we expect the exclusion criteria of Proposition~\ref{prop:reject} and~\ref{cor:refine_and_reject} for $f_1, f_2$ and $f_3$ to be satisfied for candidate solutions to the overdetermined problem: we may easily verify this in rational arithmetic for our given $\hat{E}$. Moreover, we estimate that $\alpha (g, \hat{E}) \approx .00059 ,$ thus giving that $\hat{E}$ is an approximate solution using the refinement operator $N_g.$

We remark that $\alpha $-theory does not furnish a certificate if we use the Newton fixed point system given by equation~\ref{Eq:ODNewtonCrit} in the introduction. For this system, whose defining equations are of higher degree, we estimate that $\alpha (g, \hat{E}) \approx 9.23 .$ We also found, using Gr\"{o}bner bases over a finite field, that the number of excess solutions to this fixed point system is $24,$ for a total of $34$ critical points overall for the least-squares Newton operator. By constrast, Proposition~\ref{prop:VsqDecomp} furnishes a certificate that requires no excess solutions whatsoever. Thus, even for this toy example, we see how the flexibility of square subsystems may enhance the prospects of obtaining rigorous mathematical proofs from the output of numerical computations. Although certification for overdetermined systems remains a challenge in general, similar techniques may be worth considering for problems of a larger scale.

\bibliographystyle{plain}
\bibliography{art}

\end{document}